\documentclass[12pt]{article}
\usepackage{amsmath,amsfonts,amssymb,bbm}
\usepackage{amsthm}
\usepackage{color}
\usepackage[ansinew]{inputenc}
\usepackage{epsf,epsfig,psfrag}
\usepackage{color}
\usepackage[active]{srcltx}

\DeclareMathOperator{\dist}{dist}

\DeclareMathOperator{\Proj}{proj}

\DeclareMathOperator{\Lim}{Lim}
\DeclareMathOperator{\Limsup}{Lim\ sup}
\DeclareMathOperator{\Liminf}{Lim\ inf}

\DeclareMathOperator{\M}{\mc{M}}

\begin{document}

\newtheorem{oberklasse}{OberKlasse}
\newtheorem{lemma}[oberklasse]{Lemma}
\newtheorem{proposition}[oberklasse]{Proposition}
\newtheorem{theorem}[oberklasse]{Theorem}
\newtheorem{remark}[oberklasse]{Remark}
\newtheorem*{remark*}{Remark}
\newtheorem{remarks}[oberklasse]{Remarks}
\newtheorem{remarks*}{Remarks}
\newtheorem{corollary}[oberklasse]{Corollary}
\newtheorem{definition}[oberklasse]{Definition}
\newcommand{\R}{\ensuremath{\mathbbm{R}}}
\newcommand{\N}{\ensuremath{\mathbbm{N}}}
\newcommand{\F}{\ensuremath{\mathcal{F}}}
\newcommand{\h}{\ensuremath{\mathcal{H}}}
\newcommand{\CB}{\ensuremath{\mathcal{CBC}}}
\newcommand{\wll}{\langle\!\langle}
\newcommand{\wrr}{\rangle\!\rangle}

\newcommand{\tc}{\textcolor}
\newcommand{\mc}{\mathcal}

\newcommand{\wolfj}[1]{{\color{blue}{Wolf-J\"urgen: #1}}}
\newcommand{\etienne}[1]{{\color{red}{Etienne: #1}}}
\newcommand{\janosch}[1]{{\color{red}{Janosch: #1}}}

\renewcommand{\phi}{\ensuremath{\varphi}}
\renewcommand{\epsilon}{\ensuremath{\varepsilon}}

\renewcommand{\thefootnote}{\fnsymbol{footnote}}

\title{Semilinear Parabolic Differential Inclusions \\
with One-sided Lipschitz Nonlinearities}
\author{Wolf-J\"urgen Beyn\footnotemark[1]{$ \; \; $}\footnotemark[3] \quad Etienne Emmrich\footnotemark[4]\\ 
Janosch Rieger\footnotemark[5]${ \; \; }$\footnotemark[3]}
\date{\today }

\maketitle

\abstract{We present an existence result for a partial differential 
inclusion with linear parabolic principal part and relaxed one-sided 
Lipschitz multivalued nonlinearity in the framework of Gelfand triples.
Our study uses discretizations of the differential inclusion by a 
Galerkin scheme, which is compatible with a conforming finite element method,
and we analyze convergence properties of the discrete solution sets.}

\footnotetext[1]{Department of Mathematics, Bielefeld University, Bielefeld, Germany}
\footnotetext[3]{work supported by DFG in the framework of CRC 701, project B3.}
  \footnotetext[4]{Institute of  Mathematics, Technical University of Berlin,
   Berlin, Germany \\ 
  work supported by DFG in the framework of CRC 910, project A8.}
  \footnotetext[5]{School of Mathematical Sciences, Monash University,
    Melbourne, Australia}
 {\bf Key words.} Partial differential inclusion, one-sided Lipschitz condition,
Galerkin method.\\
{\bf AMS subject classification.} primary: 35R70, 65M60; secondary: 35K20, 35K91, 49J53.
\section{Introduction} \label{sec1}
We consider the inital value problem for semilinear partial differential inclusions of the form
\begin{equation} \label{introduction:PDI}
u'(t) + Au(t) \in F(t,u(t))\ \text{for}\ t \in (0,T), \quad u(0,\cdot)=u_0,
\end{equation}
in a Gelfand triple $V\subseteq H \subseteq V^*$
with a strongly positive, linear, and bounded operator $A:V\rightarrow V^*$ 
and a genuinely set-valued nonlinearity $F:V\rightrightarrows H$ with closed,
bounded, and convex images. 
In contrast to partial differential equations with maximal monotone 
principle part (see \cite{barbu10,Deimling:85}), 
differential inclusions of type \eqref{introduction:PDI} possess a nontrivial 
solution set.
They may be considered the deterministic counterparts of stochastic 
partial differential equations since  they model deterministic 
uncertainty by a set-valued operator. 
Likewise, they provide a framework for the analysis of control systems of 
partial differential equations with control constraints (see e.g.\ 
\cite{Colonius:00,Vinter:00}).

\medskip

In the present paper, we discuss a weak reformulation 
\begin{subequations} \label{introPDI:all}
\begin{align}
&\langle u'(t),v \rangle + a(u(t),v) = (f(t),v)\quad\dot\forall\,t\in (0,T),\,\forall\,v\in V,
\label{introPDI1}\\
&f(t) \in F(t,u(t))\quad\dot\forall\,t\in (0,T),
\label{introPDI2}\\
&u(0) = u_0,   \label{introPDI:IV}
\end{align}
\end{subequations}
of inclusion \eqref{introduction:PDI}, where 
$(\cdot,\cdot):H\times H\to\R$ is the inner product in $H$, 
the duality pairing between $V^*$ and $V$ is denoted by
$\langle\cdot,\cdot\rangle:V^*\times V~\to~\R$, 
the bilinear form associated with the operator $A$ is
$a:V\times V \rightarrow~\R$, 
and the symbol '$\dot\forall$' means 'for Lebesgue-almost every'.
For an overview of existence and uniqueness results for differential inclusions
of type \eqref{introPDI:all}, we refer to \cite{Hu:Papageorgiou:00}.
Given a Galerkin scheme $(V_N)_{N \in \N}$ of finite-dimensional
subspaces of $V$, the approximate Galerkin inclusion 
\begin{subequations} \label{introGDI:all}
\begin{align}
&\langle u_N'(t),v\rangle + a(u_N(t),v) = (f_N(t),v)\quad 
\dot\forall\,t\in (0,T),\,\forall\,v\in V_N, \label{introGDI1}\\
&f_N(t)\in {F}(t,u_N(t))\quad\dot\forall\,t\in (0,T),
\label{introGDI2} \\
&u_N(0) = u_{N,0} \in V_N,   \label{introGDI:IV}
\end{align}
\end{subequations}
is compatible with standard conforming finite element approaches
for parabolic partial differential equations.

Our goal is to study the convergence of the solution set of inclusion
\eqref{introGDI:all} to the solution set of inclusion
\eqref{introPDI:all} with respect to the Hausdorff metric in $L^2(0,T;H)$.
The main novelty of our approach is to establish such a result
for a multivalued nonlinearity satisfying a relaxed one-sided Lipschitz 
property (see assumption (A4) for details).
This property, which goes back to Donchev (see e.g.\ \cite{Donchev:02}), 
is much weaker than standard Lipschitz or dissipativity conditions
treated in \cite{Hu:Papageorgiou:00}, and it found many applications 
in the theory and numerical analysis of differential inclusions, see e.g.\
\cite{Beyn:Rieger:10,Beyn:Rieger:12,DFM07,DFR07,MT16,Rieger:12a}.

\section{Problem setting and main results}
\label{sec2}
In this section we present the main results and specify the analytical setting
underlying the differential inclusion \eqref{introPDI:all}
and its Galerkin approximation \eqref{introGDI:all}.

\subsection{Preliminaries from set-valued analysis}
We refer to the monographs
\cite{Aubin:Frankowska:90} and \cite{Hu:Papageorgiou:97} for general notions 
from set-valued analysis. In the following we specify some notation that
will be  used throughout this paper.
In the following, let $X$ and $Y$ be normed spaces. 

\begin{definition}
  For any $x\in X$ and any subset $M\subseteq X$, we define the distance of
  $x$ to $M$ by
  \[ \dist(x,M)_X = \inf \{ \|x-y\|_X : y \in M \}\]
  and the proximal set by 
\[\Proj(x,M)_X:=\{y\in M: \|x-y\|_X\le\|x-z\|_X\ \forall z \in M\}.\]
\end{definition}

Recall that $\Proj(x,M)_X$ is a singleton in case $M$ is convex and
$X$ is a Hilbert space.
Moreover, by a common abuse of notation, we write
\[\|M\|_{X} := \sup_{x\in M}\|x\|_{X} \quad \text{for} \;  M\subseteq X\; \; \text{bounded.}\]
 
By $\CB(X)$ we denote the set of all closed, bounded, and convex subsets of $X$.
There are various ways of defining a topology on $\CB(X)$ which in general
are not equivalent. We will use convergence in the Hausdorff and in the
Kuratowski sense.

\begin{definition}
For any two sets $M,\widetilde{M}\subseteq X$, the 
Hausdorff semi-distance and the Hausdorff distance are defined by
\begin{equation} \label{eq:hausdorff}
\begin{aligned}
&\dist(M,\widetilde{M})_{X} 
:= \sup_{x\in M}\dist(x,\widetilde{M})_{X},&\\
&\dist_{\h}(M,\widetilde{M})_{X} 
:= \max\{\dist(M,\widetilde{M})_{X},\dist(\widetilde{M},M)_{X}\}. &
\end{aligned}
\end{equation}
\end{definition}

It is well-known that $\dist_{\h}$ defines a metric on $\CB(X)$.

\begin{definition}
A sequence $\{M_n\}_{n\in \N}$ of sets $M_n \subseteq X$ is said to converge to 
a set $M\subseteq X$ in Kuratowski sense, which is denoted by 
$\Lim_{n\rightarrow\infty}M_n=M$, if
\[\Limsup_{n\rightarrow\infty} M_n\subseteq M \subseteq \Liminf_{n\rightarrow\infty} M_n,\]
where the upper and lower Kuratowski limits of a sequence $(M_n)_{n\in\N}$ are given by
\begin{align*}
  \Limsup_{n\rightarrow\infty} M_n &:= \{x\in X: \liminf_{n\rightarrow\infty}
  \dist(x,M_n)_X=0\},\\
  \Liminf_{n\rightarrow\infty} M_n &:= \{x\in X: \lim_{n\rightarrow\infty}
  \dist(x,M_n)_X=0\}.
\end{align*}
\end{definition}

An important situation where convergence in the sense of Kuratowski implies
Hausdorff convergence is given in the following Lemma,
which is a slight variation of \cite[Chapter 7, Proposition 1.19]{Hu:Papageorgiou:97}. 

\begin{lemma} \label{Kuratowski:lemma}
Let $B\subseteq X$ be relatively compact 
and let $M_n,M\subseteq B$ for all $n\in\N$.
If $\Lim_{n\to\infty}M_n=M$, then $\dist_\h(M_n,M)_X\rightarrow 0$ as $n\to\infty$.
\end{lemma}
\begin{proof}
If $\dist(M_n,M)_X\rightarrow 0$ as $n\to\infty$ is false, there exist $\epsilon>0$ 
and a sequence $\{x_n\}_{n\in\N}$ with $x_n\in M_n$ such that $\dist(x_n,M)_X>\epsilon$ 
along a subsequence $\N'\subset\N$. 
But $x_n\in B$ for all $n\in\N$ implies that there exists a subsequence $\N''\subset\N'$ 
and $x\in\bar B$ such that $x_n\rightarrow x$ as $\N''\ni n\to\infty$. 
By the Kuratowski upper limit property, we have $x\in M$, which is a contradiction.

If $\dist(M,M_n)_X\rightarrow 0$ as $n\to\infty$ is false, there exist $\epsilon>0$ 
and a sequence $\{x_n\}_{n\in\N}\subseteq M$ such that $\dist(x_n,M_n)_X>\epsilon$ 
along a subsequence $\N'\subset\N$. 
By the relative compactness of $M$, there exists a subsequence $\N''\subseteq\N'$
such that $x_n\rightarrow x\in\bar M$ as $\N''\ni n\to\infty$. 
By the Kuratowski lower limit property, for every $n\in\N''$ there exist sequences 
$\{x^n_k\}_{k\in\N}$ with $x^n_k\in M_k$ such that $x^n_k\to x_n$ as $k\to\infty$. 
In particular, there exist $k_n\in\N$ with $k_n\to\infty$ as $n\to\infty$
and $\|x^n_{k_n}-x_n\|_X\le\frac{1}{n}$.
But then $x^n_{k_n}\to x$ as $n\to\infty$, and hence
\[\dist(x_{k_n},M_{k_n})_X\le\|x_{k_n}-x^n_{k_n}\|_X\le\|x_{k_n}-x\|_X+\|x-x^n_{k_n}\|_X
\to 0,\]
which is a contradiction.
\end{proof}

This statement will be crucial 
for proving uniform convergence of Galerkin solution sets.

We adopt the notion of measurability from  \cite[Def.~8.1.1]{Aubin:Frankowska:90}.
\begin{definition} \label{defmeasurable}
Let $F:X \rightrightarrows Y$ be a multivalued mapping with closed images.
Then $F$ is called measurable if the preimage
\[F^{-1}(M):=\{x\in X: F(x)\cap M\neq\emptyset\}\subseteq X\]
of any open set $M\subseteq Y$ is a Borel set in $X$.
\end{definition} 

In a finite-dimensional space, the following notion of upper semicontinuity 
is equivalent to the concepts in \cite{Aubin:Frankowska:90} and \cite{Deimling:92}, 
as can be shown by an elementary argument and \cite[Proposition 1.1]{Deimling:92}.
In particular, Proposition 1.4.9 from \cite{Aubin:Frankowska:90} 
and Theorem 5.2 from \cite{Deimling:92} hold for mappings with this property.

\begin{definition}
A set-valued mapping $G:\R^N\rightarrow\CB(\R^N)$ is called upper semicontinuous 
if $x=\lim_{n\rightarrow\infty}x_n$ for any sequence $(x_n)_{n\in \N}$ in
$\R^N$ implies 
\[\dist(G(x_n),G(x))_{\CB(\R^N)}\rightarrow 0\ \text{as}\ n\rightarrow\infty.\]
\end{definition}

For further definitions and elementary facts concerning  measurability of 
multivalued mappings, we refer to \cite{Aubin:Frankowska:90}.
The present paper deviates from this source inasmuch as a multivalued mapping
$F:X\rightarrow\CB(Y)$ will be called continuous at $x\in X$ if for any sequence
$\{x_n\}_{n\in\N}\subseteq X$ with $\lim_{n\to\infty}\|x_n-x\|_X=0$ we have
\[\dist_\h(F(x),F(x_n))_Y \rightarrow 0\ \text{as}\ n\rightarrow\infty.\]

\subsection{Preliminaries from the theory of differential inclusions}

The following result is an adapted version of the existence theorem 
\cite[Theorem 5.2]{Deimling:92} for the initial value problem for a finite-dimensional differential inclusion
\begin{equation} \label{diffincl}
u'\in G(t,u) \quad \dot\forall\,t\in(0,T), \quad u(0)=u_0.
\end{equation}
In this setting, solutions are elements of the space $\mathrm{AC}([0,T];\R^N)$
of absolutely continuous functions.

\begin{theorem} \label{existencemultiODE}
Let $G:[0,T]\times\R^N\rightarrow\CB(\R^N)$ be a set-valued map with the following properties:
\begin{itemize}
\item[(i)] The mapping $t\mapsto G(t,v)$ is measurable for all $v \in \R^N$.
\item[(ii)] The mapping $v\mapsto G(t,v)$ is upper semicontinuous for almost every $t\in(0,T)$.
\item[(iii)] There exists $K\ge 0$ such that all solutions
$u\in\mathrm{AC}([0,T'];\R^N)$ of inclusion \eqref{diffincl} on $(0,T')$ with $0<T'\le T$
satisfy the a priori bound $\|u\|_{L^{\infty}(0,T';\R^N)}\le K$.
\end{itemize}
Then there exists a global solution $u\in\mathrm{AC}([0,T];\R^N)$ of \eqref{diffincl}.
\end{theorem}
Note that Theorem \ref{existencemultiODE} differs from \cite[Theorem 5.2]{Deimling:92}, 
which requires a linear bound 
\[\|G(t,v)\| \le c(t)(1+\|v\|)\quad\dot\forall\,t\in(0,T),\,\forall\,v\in\R^N.\]
This bound is, however, exclusively used to prove a-priori bound (iii),
which we assume in our setting. 

\medskip

We also use the following refined version of Gronwall's Lemma.

\begin{lemma} \label{extended:Gronwall}
Let $s\in\mathrm{AC}([0,T];\R_+)$, let $\kappa\in L^1((0,T);\R)$, 
and let $\rho\in L^1((0,T);\R_+)$ be such that
\begin{equation} \label{bad:scalar:ODE}
s(t)s'(t) \le \kappa(t)s(t)^2+\rho(t)s(t)\quad\dot\forall\,t\in(0,T).
\end{equation}
Then we have
\begin{align} 
&s'(t) \le \kappa(t)s(t)+\rho(t)\quad\dot\forall\,t\in(0,T) \label{desired:scalar:ODE},\\
&s(t) \le s(0)e^{\int_0^t\kappa(\tau)d\tau} + \int_0^te^{\int_{\sigma}^t\kappa(\tau)d\tau}\rho(\sigma)d\sigma\quad\forall\,t\in[0,T].
\label{extended:Gronwall:bound}
\end{align}
\end{lemma}
\begin{proof}
Consider the set $Z:=\{t\in [0,T]: s(t)=0\}$, and let $\tilde Z\subset Z$ 
be the subset of all density points of $Z$. Then $Z$ and $\tilde Z$
have the same Lebesgue measure due to
the Lebesgue density theorem \cite[Theorem 2.2.1]{Kannan:Krueger:96}.
If $t\in [0,T]\setminus Z$, then inequality \eqref{desired:scalar:ODE} holds,
because both sides of inequality \eqref{bad:scalar:ODE} can be divided by $s(t)$. 
If $t\in\tilde Z$ and $s'(t)$ exists,
then $s'(t)=0$ and inequality \eqref{desired:scalar:ODE} holds, because $s(t)=0$ 
and $\rho(t)\ge 0$. 
Thus $s'$ satisfies inequality \eqref{desired:scalar:ODE} almost everywhere in $(0,T)$, 
and the Gronwall Lemma yields the estimate \eqref{extended:Gronwall:bound}.
\end{proof}

\subsection{Function spaces}
\label{sec2.2}
Our standing assumptions on the underlying spaces are as follows.
\begin{itemize}
\item[(S1)] Let $(V,\|\cdot\|_V)$ be a separable Hilbert space,
densely and compactly embedded in a Hilbert space $(H,(\cdot,\cdot),\|\cdot\|_H)$.
\end{itemize}
In particular, there exists $c_{VH}>0$ such that
\[\|v\|_H\le c_{VH}\|v\|_V\quad\forall\,v\in V.\]
The dual space $(V^*,\|\cdot\|_{V^*})$ of $V$ is equipped with the norm
\[\|f \|_{V^*}=\sup_{\|v\|_V=1}|\langle f, v \rangle|,\] 
where $\langle\cdot,\cdot\rangle:V^*\times V\to\R$ is the duality pairing.
We identify $H$ with its dual so that $V\subseteq H\subseteq V^*$ 
form a Gelfand triple with 
\begin{align*}
\langle u,v\rangle=(u,v)\quad\forall\,u\in H,\,v\in V.
\end{align*}
Our assumptions on the Galerkin scheme for the inclusion
\eqref{introPDI:all} are summarized below.
\begin{itemize}
\item[(S2)] Let $\{V_N\}_{N\in\N}$
be a nested sequence of finite-dimensional subspaces of $V$ such that 
for all $v \in  V$
\begin{equation*}
\dist(v,V_N)_{V} \rightarrow 0 \; \text{as} \; N \rightarrow 
\infty.
\end{equation*}
\item[(S3)] There exists $C_P>0$ such that the $H$-orthogonal projection 
$P_N:H\rightarrow V_N$ onto $V_N$ satisfies $\|P_Nv\|_V\le C_P\|v\|_V$ for all $v\in V$ and $N\in\N$.
\end{itemize}
Condition (S3) is crucial for the convergence result in 
Theorem \ref{upper:Kuratowski}. It holds for various types of 
finite element spaces under suitable geometric conditions, see \cite{carstensen02,emmrich:siska:12} and the references therein.

\medskip

For a Banach space $X$, we denote the spaces of all Bochner measurable 
and all continuous functions from $[0,T]\subset\R$ to $X$ by  $\M(0,T;X)$
and $\mc{C}([0,T],X)$, respectively. 
We refer to \cite{Diestel:Uhl:98,Roubicek:05} for the general theory of 
Bochner-Lebesgue spaces.
Several equivalent notions of measurability are discussed in
\cite[p.93, Proposition 12]{dinculeanu67} and \cite[p.42]{Diestel:Uhl:98}.
By $L^r(0,T;X) (1 \le r < \infty)$ we denote the functions in $\M(0,T;X)$ with finite norm
\begin{equation*}
\|u\|_{L^r(0,T;X)} = (\int_0^T \|u(t)\|_X^r dt)^{1/r} .
\end{equation*}
In case $X=\R$ we write $L^r(0,T)$ instead of $L^r(0,T;\R)$.
A function $u \in L^1_{\mathrm{loc}}(0,T;V^*)$ has a weak derivative
$u' \in L^1_{\mathrm{loc}}(0,T;V^*)$ provided
\begin{equation*} 
\int_0^T u'(t) \varphi(t) dt = - \int_0^T u(t) \varphi'(t) dt 
\quad\forall\,\varphi\in C_c^{\infty}(0,T),
\end{equation*}
where $\varphi\in C_c^{\infty}(0,T)$ denotes the space of all infinitely many times 
continuously differentiable functions on $(0,T)$ with compact support.
This relation may be written equivalently as
\begin{equation} 
\int_0^T\!\!\!\varphi(t) \langle u'(t), v\rangle  dt =
-\!\!\int_0^T\!\!\!\varphi'(t)\langle u(t),v\rangle dt\ \,
\forall\,v\in V,\varphi\in C_c^{\infty}(0,T).
\label{weakderiv2}
\end{equation} 
Following \cite[Lemma 19.1]{Tartar06}, we introduce the space
\begin{equation*}
W_+=L^2(0,T;V^*)+ L^1(0,T;H),
\end{equation*}
with norm (cf. \cite[Kap.IV, \S 1]{Gajewski:Groeger:Zacharias:74})
\begin{align*}
\|f\|_{W_+}= \inf\{\max(\|g\|_{L^2(0,T;V^{*})},\|h\|_{L^1(0,T;H)}): f=g+h&\\
g \in L^2(0,T;V^{*}),\,h \in L^1(0,T;H)\}&.
\end{align*}
The dual of $W_+$ can be identified with 
\begin{equation*}
W_+^* = L^2(0,T;V) \cap L^{\infty}(0,T;H)
\end{equation*}
equipped with the norm
\begin{equation*}
\|v\|_{W_+^*}= \|v\|_{L^{\infty}(0,T;H)}+ \|v\|_{L^2(0,T;V)}
\end{equation*}
and duality pairing $\wll\cdot,\cdot\wrr:W_+^*\times W_+\to\R$ given by
\begin{equation} \label{dualitypairing}
\wll v, f \wrr
=\int_0^T \Big( \langle g(t), v(t) \rangle + (h(t),v(t)) \Big) dt
\end{equation}
(cf. \cite[Kap.I, \S 5 and Kap.IV, \S 1]{Gajewski:Groeger:Zacharias:74}),
where $v \in W_+^*$ and $f=g+h$ with $g \in L^2(0,T;V^{*})$ and $h \in L^1(0,T;H)$.
We look for solutions of
\eqref{introPDI:all} in the space
\begin{equation*} 
W= \{ u \in L^2(0,T;V): u' \; \text{exists and lies in}\; W_+ \}
\end{equation*}
with norm  given by 
\begin{equation}\label{defnormW}
\|u\|_{W} = \|u\|_{L^2(0,T;V)}+ \|u'\|_{W_+}.
\end{equation}
Indeed, one may show that
\begin{equation} \label{repW}
  W= \{ u\in  W_+^*: u' \in W_+ \}.
  \end{equation}
From \cite[Lemma19.1, p.114]{Tartar06} one further obtains the continuous embedding
\begin{equation} \label{embedW}
W \subseteq C([0,T],H), 
\end{equation}
which shows that the initial condition \eqref{introPDI:IV} makes sense. 
Moreover, by the Lions-Aubin Theorem (see \cite[Lemma 7.7]{Roubicek:05}) and the 
compact embedding of $V$ in $H$, we have for all $1 \le r < \infty$ 
\begin{equation} \label{compactembedW}
W \; \text{is compactly embedded into} \; L^r(0,T;H).
\end{equation}

Next we consider function spaces for the Galerkin approximations.
Since $V_N$ is finite dimensional we need not distinguish topologies in
the image spaces and choose our solutions to be in the space of 
absolutely continuous functions
\[W_N:=\mathrm{AC}([0,T];V_N).\]
Note that $u_N\in W_N$ implies $u_N'(t) \in V_N$ for almost every $t\in(0,T)$ and
\begin{equation} \label{diffsquare}
\frac{1}{2} \frac{d}{dt} \|u_N(t)\|_H^2 = (u'_N(t),u_N(t))\quad\dot\forall t\in(0,T).
\end{equation}
Equation \eqref{introGDI1} may now be written as
\begin{equation} \label{VNcomment2}
(u_N'(t),v)+a(u_N(t),v) = (f_N(t),v)\quad\ \dot\forall t\in(0,T),\ \forall\,v\in V_N.
\end{equation}

\subsection{Problem data}
\label{sec2.3}

We state our main assumptions on $a$ and $F$.
\begin{itemize}
\item [(A1)] The bilinear form $a(\cdot,\cdot):V\times V\rightarrow\R$ is strongly positive and bounded, i.e.\ there
exist constants $c_a, C_a>0$ such that for all $v,w\in V$
 \begin{equation} \label{positivebounded}
c_a\|v\|_{V}^2 \le a(v,v) \quad \text{and} \quad
 a(v,w)\le C_a\|v\|_V\| w \|_V.
\end{equation}
\item [(A2)] The set-valued mapping ${F}:[0,T]\times V\rightarrow\CB(H)$ 
  is Ca\-ra\-th\'{e}o\-do\-ry, i.e.\ the mapping
  $t\mapsto F(t,v):[0,T]\rightarrow \CB(H)$ is  
measurable for any $v\in V$, and for almost every $t\in(0,T)$, the mapping
$v\mapsto F(t,v):(V,\|\cdot\|_V) \rightarrow(\CB(H),\dist_\h(\cdot,\cdot)_H)$
is continuous.
\item [(A3)] There exist a function $\alpha\in L^1(0,T)$ and a constant $c_F>0$ 
such that for almost every $t\in(0,T)$ and all $u,v \in V$, we have bounds
\begin{align*} 
&\|F(t,0)\|_H\le\alpha(t),\\
&\dist_{\h}(F(t,u),F(t,v))_{H}\le c_F(1+\|u\|_V+\|v\|_V)\|u-v\|_H .
\end{align*}
\item[(A4)] The mapping $F$ is relaxed one-sided Lipschitz in its second argument, i.e.\ there exists $\ell\in L^1(0,T)$, 
such that for almost every $t\in(0,T)$, all $v,\tilde v\in V$, and all $g\in{F}(t,v)$,
there exists some $\tilde g\in{F}(t,\tilde v)$ such that
\[(g-\tilde g,v-\tilde v) \le \ell(t)\|v-\tilde v\|_H^2.\]
\item [(A5)] The initial values satisfy $u_{N,0}=P_Nu_0$ for all $N\in\N$.
\end{itemize}
Note that the Lipschitz condition in (A3) is of local type and
implies a stronger continuity property of $v\mapsto F(t,v)$ than (A2),
namely
\begin{align*}
v,v_k\in V, &\; \|v_k\|_V \le C (k \in \N),\; \|v_k-v\|_H \rightarrow 0
\quad \text{as}\quad k\rightarrow \infty \\ 
\Longrightarrow  & \quad \dist_\h(F(t,v_k),F(t,v))_H\rightarrow 0 
\quad \text{as}\quad k\rightarrow \infty.
\end{align*}
Moreover, condition (A3) with $u=0$  implies the growth estimate
\begin{equation} \label{L1bound}
\|{F}(t,v)\|_{H}\le\alpha(t)+c_F (1+\|v\|_V)\|v\|_H  \quad \dot\forall t\in(0,T),\ \forall v\in V.
\end{equation}

\subsection{Main results}

The definition of solutions to \eqref{introPDI:all} is straightforward.

\begin{definition}
A function $u\in W$ is called a solution of \eqref{introPDI:all}
 (or a weak solution of \eqref{introduction:PDI}) if $u$ satisfies
\eqref{introPDI:IV} and there exists 
$f \in L^1(0,T;H)$ satisfying  \eqref{introPDI1}, \eqref{introPDI2} 
for almost every $t \in [0,T]$.
\end{definition}
Similarly, we define a weak solution of \eqref{introGDI:all}.
\begin{definition} Let $N \in \N$ be fixed.
A function $u_N\in W_N$ is called a solution of \eqref{introGDI:all}, 
 (or a Galerkin solution of 
\eqref{introPDI:all})
if $u$ satisfies \eqref{introGDI:IV} and there exists some 
$f_N\in L^1(0,T;H)$ satisfying 
\eqref{introGDI1}, \eqref{introGDI2} for almost every $t\in [0,T]$. 
\end{definition}
Our first result is a uniform a priori bound for solutions of the differential
inclusion and of the  Galerkin inclusion.
\begin{proposition} \label{aprioribound}
There exists a constant $K>0$ such that all solutions
$u\in W$ and $u_N\in W_N$ of inclusions \eqref{introPDI:all} and \eqref{introGDI:all} satisfy 
\begin{equation} \label{estuNu}
\|u\|_W\le K,\quad\|u_N\|_{W}\le K\quad\forall\,N\in\N.
\end{equation}
The constant $K$ depends only on the values $\|u_0\|_H$ and
$\sup_{N\in\N}\|u_{0,N}\|_H$ and on the bounds from (S1)-(S3) and (A1)-(A4).
\end{proposition}
The proof will show that the same bounds hold for all
solutions of inclusions \eqref{introPDI:all} and \eqref{introGDI:all} on any subinterval $[0,T']$ with $0<T'\le T$.
Note also that the bound in $W$  implies a bound in $L^{\infty}(0,T;H)$
according to \eqref{repW} and \eqref{embedW}.

From now on we denote by $\mc{S}=\mc{S}(u_0)$ the set of all solutions of \eqref{introPDI:all} and by $\mc{S}_N=\mc{S}_N(P_N u_0)$ the set of all
solutions of \eqref{introGDI:all} with initial data given by (A5). By
Proposition \ref{aprioribound} and (S2) these sets are uniformly bounded for
$u_0$ fixed.

For our second main result, we prove existence of solutions to \eqref{introGDI:all}
and use the bounds from Proposition \ref{aprioribound} to extract a weakly convergent
subsequence.
The limit of this sequence is a solution to \eqref{introPDI:all}.
This is essentially sufficient to conclude convergence of $\mc{S}_N$  
to $\mc{S}$ in the upper Kuratowski sense in $L^2(0,T;H)$. 

\begin{theorem} \label{upper:Kuratowski}
The solution sets $\mc{S}$ and $\mc{S}_N$ for all $N\in\N$ are non\-emp\-ty, 
and we have $\Limsup_{N\to\infty}\mc{S}_N\subset\mc{S}$ in $L^2(0,T;H)$.
\end{theorem}

Our second main result may be viewed as convergence in the lower Kuratowski 
sense in $W_+^*$. 

\begin{theorem} \label{lower:Kuratowski}
For every solution $u\in\mc{S}$, there exists a sequence $\{u_N\}_{N\in\N}$ with
$u_N\in\mc{S}_N$ for all $N\in\N$ and 
\[\|u-u_N\|_{W_+^*} \rightarrow 0\quad\text{as}\quad N\rightarrow\infty.\]
\end{theorem}

As a consequence of Theorems \ref{upper:Kuratowski} and \ref{lower:Kuratowski},
we obtain the next main result of this paper.

\begin{theorem} \label{main:theorem}
The sets $\mc{S}_N$ converge to $\mc{S}$ in the sense that
\[\dist_{\h}(\mc{S},\mc{S}_N)_{L^2(0,T;H)}\rightarrow 0\quad\text{as}\quad N\rightarrow\infty.\]
\end{theorem}

\section{A priori estimates}

Our first step is an a priori bound for solutions of the differential inclusions 
\eqref{introPDI:all}
and \eqref{introGDI:all}.

\begin{proof}[Proof of Proposition \ref{aprioribound}]
 Recall that estimate \eqref{estuNu} requires to show 
\begin{equation} \label{estuNprime}
  \| u_N \|_{L^2(0,T;V)} \le K_0, \quad
  \| u'_N \|_{L^2(0,T;V^*)+ L^1(0,T;H)} \le K'_0
\end{equation}
for suitable constants $K_0,K'_0$ depending on $C_0:=\sup_{N\in \N}\|u_{N,0}\|_H$.
If $u_N\in \mathrm{AC}([0,T];V_N)$ solves inclusion \eqref{introGDI:all}, 
then according to \eqref{VNcomment2}, 
there exists a function $f_N \in L^1(0,T;H)$ with
\begin{align}
&(u'_N(t),v)+a(u_N(t),v)=(f_N(t),v)\quad\dot\forall t\in(0,T),\,\forall v\in V_N,
\label{weakfN}\\
&f_N(t) \in F(t,u_N(t))\quad\dot\forall t\in(0,T). \label{finclude}
\end{align}
Setting $v =u_N(t)$ and using (A1) and \eqref{diffsquare}
gives the energy estimate
\begin{equation}\label{energy1est}
\frac{1}{2} \frac{d}{dt} \|u_N(t)\|_H^2 + c_a \|u_N(t)\|_V^2 \le 
 (f_N(t),u_N(t))\quad\dot\forall\,t\in (0,T).
\end{equation}
Next we apply condition (A4) for fixed $t\in[0,T]$ 
with $v=u_N(t)$, $\tilde{v}=0$, $g=f_N(t)$
and find an element $\tilde{g} \in F(t,0)$ such that
\[( f_N(t)-\tilde{g},u_N(t) ) \le \ell(t) \|u_N(t)\|_H^2.\]
When combined with \eqref{L1bound} and \eqref{energy1est}, we obtain
\begin{equation} \label{energy2est}
\frac{1}{2} \frac{d}{dt} \|u_N(t)\|_H^2 + c_a \|u_N(t)\|_V^2 \le
\ell(t) \|u_N(t)\|_H^2 + \alpha(t) \|u_N(t)\|_H.
\end{equation}
Then the Gronwall Lemma \ref{extended:Gronwall} leads to the
estimate for $t \in [0,T]$
\begin{equation} \label{Gronwall1} 
\begin{aligned}
\|u_N(t)\|_H \le & \|u_{N,0}\|_He^{\int_0^{t}\ell(\tau) d\tau}
+ \int_0^t\alpha(\sigma)e^{\int_{\sigma}^t \ell(\tau) d\tau}d\sigma \\ 
\le & e^{\|\ell\|_{L^1(0,T)}} ( C_0 + \| \alpha \|_{L^1(0,T)} )
=: K_1.
\end{aligned}
\end{equation}
This proves $\|u_N\|_{L^{\infty}(0,T;H)} \le K_1$.  
In the next step we use this estimate and integrate \eqref{energy2est} to obtain
\begin{equation} \label{L2vestimate}
c_a \|u_N\|^2_{L^2(0,T;V)} \le \frac{C_0^2}{2} + K_1(K_1\|\ell\|_{L^1[0,T]}
+ \| \alpha \|_{L^1[0,T]})=:C_1,
\end{equation} 
so that the first part of \eqref{estuNprime} follows with $K_0^2=\frac{C_1}{c_a}$.
We use the duality relation 
\begin{equation*}
\| u'_N\|_{W_+}= \sup_{\varphi \in W_+^*,\,\|\varphi\|_{W_+^*=1}}\wll \varphi, u'_N \wrr
\end{equation*}
to estimate the derivative $u'_N \in L^1(0,T;V_N)$.
Because of \eqref{dualitypairing} and \eqref{weakfN}, we find 
\begin{equation} \label{energy3est}
\begin{aligned}
\wll\varphi, u'_N \wrr = & 
 \int_0^T (\varphi(t),u'_N(t)) dt 
= \int_0^T (u'_N(t),P_N \varphi(t)) dt \\
= & - \int_0^T a(u_N(t),P_N \varphi(t))dt + \int_0^T(f_N(t),P_N \varphi(t))dt.
\end{aligned}
\end{equation}
By \eqref{L1bound} and \eqref{L2vestimate}, we arrive at the estimate
\begin{equation} \label{fNest}
  \begin{aligned}
    \| f_N \|_{L^1(0,T;H)} \le& \|\alpha\|_{L^1(0,T)} +
    c_F(\sqrt{T}+ \|u_N\|_{L^2(0,T;V)})\|u_N\|_{L^2(0,T;H)}\\
    \le& \|\alpha\|_{L^1(0,T)}+c_Fc_{VH}(\sqrt{T}+K_0)K_0=:C_2.
    \end{aligned}
    \end{equation}
Further, using (A1) and $\|P_N \varphi(t) \|_V \le C_P \| \varphi(t)\|_V$ from 
(S3), we deduce from \eqref{energy3est} and \eqref{fNest} that
\begin{equation*}
\begin{aligned}
  |\wll\varphi,u'_N\wrr|  \le & 
  \int_0^T\!\!\! C_a C_P \|u_N(t)\|_V  \|\varphi(t)\|_V dt 
  \!+\! \int_0^T\!\! \|f_N(t)\|_H \|P_N \varphi(t)\|_H dt \\
\le & C_a C_P  K_0 \|\varphi\|_{L^2(0,T;V)}+ 
C_2 \|\varphi\|_{L^{\infty}(0,T;H)} 
\le K'_0 \|\varphi\|_{W_+^*}
\end{aligned}
\end{equation*}
with $K_0':=C_aC_PK_0+C_2$, which proves our assertion.

\medskip

The bound for $\mc{S}$ can be obtained essentially in the same way,
because according to \cite[Ch. 20]{Tartar06}, functions $u \in W$ satisfy
\begin{equation*}
\frac{1}{2} \frac{d}{dt} \|u(t)\|_H^2 = \langle u'(t),  u(t) \rangle
\quad\dot\forall\,t\in (0,T).
\end{equation*}
\end{proof}

\section{Existence of Galerkin solutions}

The following Filippov-type result measures the minimal distance from a given $v_N\in W_N$
to $\mc{S}_N$.
It is convenient to introduce the operator
\[A_N:V_N\to V_N,\quad (A_Nv,w)=a(v,w)\quad\forall\,v,w\in V_N.\]

\begin{proposition} \label{Filippov}
For any $v_N\in W_N$, any $u_{N,0}\in V_N$, and any function $\delta_N\in L^1(0,T)$ 
with
\begin{equation*}
\dist(v_N'(t)+A_Nv_N(t),F(t,v_N(t)))_H\le\delta_N(t)\quad\dot\forall\,t\in(0,T),
\end{equation*}
there exists a solution $u_N\in W_N$ of \eqref{introGDI:all} satisfying
\begin{equation} \label{estfilippov1}
\|v_N-u_N\|_{L^\infty(0,T;H)}
\le C_{\ell}(\|v_N(0)-u_{N,0}\|_H+\|\delta_N\|_{L^1(0,T)}),
\end{equation}
\begin{equation}
  \mbox{}\hspace{-2ex}
  \begin{aligned}
    c_a &\|v_N-u_N\|_{L^2(0,T;V)}^2 \le \|\delta_N\|_{L^1(0,T)} |v_N-u_N\|_{L^\infty(0,T;H)}
    \\+ &
\|\ell_+\|_{L^1(0,T)}\|v_N-u_N\|_{L^\infty(0,T;H)}^2 +
 \|v_N(0)-u_{N,0}\|_H^2,
\end{aligned}
\label{estfilippov2}
\end{equation}
with $C_{\ell} = \sup_{0\le s \le t \le T}\exp\left(\int_s^t \ell(\tau) d\tau\right)$
and $\ell_+(t) = \max(0,\ell(t))$.
\end{proposition}

\begin{proof}
We construct a multivalued right-hand side in such a way that any
solution of the corresponding differential inclusion is also 
a solution of the Galerkin inclusion \eqref{introGDI:all}
and satisfies bounds \eqref{estfilippov1} and \eqref{estfilippov2}.
Then we invoke Theorem \ref{existencemultiODE} to ensure the 
existence of such a solution.

Since $V_N$ is finite-dimensional, there exists $\beta_N>0$ such that 
$\|v\|_V\le\beta_N\|v\|_H$ for all $v\in V_N$.
Hence the 
operator $A_N:V_N\to V_N$
is continuous with bound $\|A_Nv\|_H\le\beta_N C_a\|v\|_V$.
By \cite[Theorem 8.2.8]{Aubin:Frankowska:90}, the mapping $t\mapsto F(t,v_N(t))$ 
is measurable, and \cite[Corollary 8.2.13]{Aubin:Frankowska:90} guarantees
measurability of the function
\[f:[0,T]\to H,\quad f(t):=\Proj(v_N'(t)+A_Nv_N(t),F(t,v_N(t)))_H.\]

Consider the mapping $L_N:[0,T]\times V_N\rightrightarrows H$ given by
\[L_N(t,v):=\{g\in H: (f(t)-g,v_N(t)-v)\le\ell(t)\|v_N(t)-v\|_H^2\}.\]
By \cite[Theorem 8.2.9]{Aubin:Frankowska:90}, the mapping $t\mapsto L_N(t,v)$ 
is measurable for all $v\in V$, and it is easy to see that $v\mapsto L_N(t,v)$ 
has closed graph in $(V_N,\|\cdot\|_V)\times(H,\|\cdot\|_H)$ for almost every
$t\in(0,T)$.
In particular, its images are closed, and they are convex by construction.

\medskip

Now consider the map $P_NF:[0,T]\times V_N\rightrightarrows V_N$.
Since the projection $P_N$ is linear and has operator norm $1$ w.r.t. $\| \cdot \|_H$, the mapping $P_NF$ 
has convex images and inherits growth bound \eqref{L1bound} and continuity in $v$ from $F$.
By \cite[Theorem 8.2.8]{Aubin:Frankowska:90}, the mapping $t\mapsto P_NF(t,v)$ 
is measurable for all $v\in V_N$.

To show that $P_NF(t,v)$ is $H$-closed for almost every $t\in(0,T)$ and all $v\in V_N$,
fix $t\in(0,T)$ and $v\in V_N$  and consider a sequence
$\{g_k\}_{k\in\N}\subset P_NF(t,v)$ 
with $\lim_{k\to\infty}\|g_k-g\|_H=0$ for some $g \in H$.
There exists$\{f_k\}_{k\in\N}\subset F(t,v)$ such that $g_k=P_Nf_k$ for all $k\in\N$.
By Mazur's theorem, and since $F(t,v)\in\mc{CBC}(H)$, there exist $\tilde{f}\in F(t,v)$ and 
a subsequence $\N'\subset\N$ such that $f_k\rightharpoonup \tilde{f}$ in $H$ as $\N'\ni k\to\infty$.
As a consequence, we have $g_k=P_Nf_k\rightharpoonup P_N\tilde{f}$ in $H$ as $\N'\ni k\to\infty$,
so uniqueness of the weak limit yields  $g=P_N\tilde{f}\in P_NF(t,v)$.

Since $\dim V_N<\infty$, the images $P_NF(t,v)$ are compact in $H$ for almost every 
$t\in(0,T)$ and all $v\in V_N$.

\medskip

Finally, the mapping
\[F_N:[0,T]\times V_N\rightrightarrows H,\quad F_N(t,v):=P_NF(t,v)\cap L_N(t,v),\]
has nonempty images for almost every $t\in(0,T)$ according to assumption (A4).
By the above, we have $F_N(t,v)\in\mc{CBC}(H)$ for almost every $t\in(0,T)$ 
and every $v\in V_N$, and $F_N$ inherits growth bound \eqref{L1bound} from $P_NF$.
By \cite[Theorem 8.2.4]{Aubin:Frankowska:90}, the mapping $t\mapsto F_N(t,v)$ 
is measurable for all $v\in V$, and by \cite[Proposition 1.4.9]{Aubin:Frankowska:90},
the mapping $v\mapsto F_N(t,v)$ is upper semicontinuous for almost every $t\in(0,T)$.

\medskip

In particular, the right-hand side of the differential inclusion
\begin{equation} \label{fdodi}
u_N'(t)\in F_N(t,u_N(t))-A_Nu_N(t)\quad\dot\forall\,t\in(0,T),\quad u_N(0)=u_{N,0},
\end{equation}
is measurable in time and upper semicontinuous in the second argument with compact and convex images
in $(V_N,\|\cdot\|_H)$.
For any $u_N\in \mathrm{AC}([0,T];V_N)$ solving inclusion \eqref{fdodi}, the function
\[f_N:[0,T]\to H,\quad f_N(t):=u_N'(t)+A_Nu_N(t),\]
is an element of $ L^1(0,T;H)$. 
Hence $u_N$ and $f_N$ solve the Galerkin inclusion \eqref{introGDI:all},
and Proposition \ref{aprioribound} provides an a priori bound for $u_N$.

Now all assumptions of Theorem \ref{existencemultiODE} are satisfied, so there exists 
indeed a solution $u_N\in \mathrm{AC}([0,T];V_N)$ of the inclusion \eqref{fdodi}, and hence
of inclusion \eqref{introGDI:all}.
Using the definition of $f$ and $L_N$ we
estimate the distance between $u_N$ and $v_N$ as follows
\begin{align}
& \frac{1}{2}\frac{d}{dt}\|v_N(t)-u_N(t)\|_H^2 + c_a \|v_N(t)-u_N(t)\|_V^2\nonumber\\
&\le(v'_N(t) - u'_N(t), v_N(t) - u_N(t))
+ a(v_N(t)-u_N(t),v_N(t)-u_N(t))\nonumber\\
&=\left((v_N'(t)+A_Nv_N(t)-f(t))+(f(t)-f_N(t)),v_N(t)-u_N(t)\right)\nonumber\\
&\le \delta_N(t) \|v_N(t)-u_N(t)\|_H + \ell(t)\|v_N(t)-u_N(t)\|^2_H.\label{energygeneral}
\end{align}
By Lemma \ref{extended:Gronwall}, we obtain 
\begin{equation*} 
\|v_N(t)-u_N(t)\|_H \le e^{\int_0^t\ell(s)ds}\|v_N(0)-u_N(0)\|_H
+\int_0^t\!e^{\int_s^t \ell(\tau) d\tau}\delta_N(s)ds,
\end{equation*}
which yields \eqref{estfilippov1}.
Integrating inequality \eqref{energygeneral} yields estimate \eqref{estfilippov2}.
\end{proof}

\section{Convergence of solution sets}

In the following, we prove upper and lower Kuratowski convergence of the
sets $\mc{S}_N$ of Galerkin solutions to the set $\mc{S}$, from which we
deduce Hausdorff convergence.

\subsection{Upper limit of Galerkin solution sets}

In this section, we show that the set $\mc{S}$ of all exact solutions
contains the upper Kuratowski limit of the sets $\mc{S}_N$ in $L^2(0,T;H)$.

\begin{proof}[Proof of Theorem \ref{upper:Kuratowski}]
\emph{a) The sets $\mc{S}_N$ are nonempty.}\\
This follows from Proposition \ref{Filippov} applied to the 
function $v_N=0$, since (A3) implies
\[\delta_N(t):=\dist(0,F(t,0))_H\le\|F(t,0)\|_H\le\alpha(t).\]
\noindent\emph{b) Extraction of convergent subsequences.}\\
Let $\{u_N\}_{N\in\N}$ be a sequence with $u_N\in\mc{S}_N$ for all $N\in\N$.
Consider a subsequence $\N'\subset\N$, and let
$f_N\in L^1(0,T;H)$ satisfy \eqref{introGDI1} and \eqref{introGDI2} for all $N\in\N'$. By Proposition \ref{aprioribound} the sequence $(u_N)_{N \in \N}$ is
bounded in $W$. Therefore,
by the compact embedding \eqref{compactembedW}, there exist
a subsequence $\N''\subseteq\N'$ and some $u\in L^2(0,T;H)$ satisfying
\begin{equation} \label{subsequence:converge}
\|u_N-u\|_{L^2(0,T;H)}\to 0\quad\text{as}\quad\N''\ni N\to\infty.
\end{equation}
According to Proposition \ref{aprioribound},
there exists a subsequence $\N'''\subseteq \N''$
such that $(u_N)_{N\in \N'''}$ is weak-$*$ convergent 
in $L^2(0,T;V)\cap L^{\infty}(0,T;H)$ (cf. \cite[Ch.21.8]{Zeidler:2A}). 
Due to the uniqueness of the weak-$*$ limit, we have
\begin{equation} \label{uNweakstar}
u_N \overset{*}{\rightharpoonup} u \in L^2(0,T;V) \cap L^{\infty}(0,T;H)\quad
\text{as}\quad\N'''\ni N\rightarrow\infty.
\end{equation}
The estimates \eqref{L1bound} and \eqref{estuNu}
show that $\{f_N\}_{N\in\N'''}$ is bounded in $L^1(0,T;H)$. 
By the Dunford-Pettis theorem \cite[Thm. IV 2.1, p. 101]{Diestel:Uhl:98}
the convergence \eqref{subsequence:converge} ensures uniform integrability
with respect to the Lesbesgue measure $\mu$, i.e. for any Lebesgue measurable
set $E \subseteq (0,T)$ one has
\[\sup_{N\in\N'''}\int_E\|u_N(t)\|_H^2dt\to 0\quad\text{as}\quad\mu(E)\to 0.\]
Using bounds \eqref{L1bound} and \eqref{estuNu}, we obtain
for $\mu(E)\to 0$
\begin{align*}
&\sup_{N\in\N'''}\|\int_Ef_N(t)dt\|_H
\le\sup_{N\in\N'''}\int_E\Big(\alpha(t)+c_F(1+\|u_N(t)\|_V)\|u_N(t)\|_H\Big)dt\\
&\le\int_E \alpha(t)dt+c_F(\sqrt{T}+K)\sup_{N\in\N'''}(\int_E\|u_N(t)\|_H^2
   dt)^\frac12
\to 0.
\end{align*}
 Therefore, the set $\{f_N:N \in \N''' \}$ is uniformly
integrable and bounded in $L^1(0,T;H)$, and balls in $H$ are relatively
weakly compact.
Then, again by the Dunford-Pettis theorem, there exists a subsequence 
$\N^{(4)}\subset\N'''$ such that
\begin{equation} \label{fweakconv}
f_N\rightharpoonup f\ \text{in}\ L^1(0,T;H)\quad\text{as}\quad\N^{(4)}\ni N\rightarrow 
\infty.
\end{equation}

\noindent\emph{c) The limits solve the differential inclusion.}\\
By assumption (A2) and estimate \eqref{L1bound}, the sets
\begin{align*}
&\F:=\{g\in L^1(0,T;H):g(t)\in F(t,u(t))\ \dot\forall\,t\in [0,T]\},\\
&\F_\epsilon:=\{g\in L^1(0,T;H): \dist(g,\F)_{L^1(0,T;H)} \le 
\epsilon\},
\end{align*}
are closed, bounded and convex subsets of $L^1(0,T;H)$. Note that
the map $t \mapsto F(t,u(t))$ is measurable due to \cite[Theorem 8.2.8]{Aubin:Frankowska:90}, so that measurability of the map $t \mapsto \dist(g(t),F(t,u(t))$ follows from \cite[Corollary 8.2.13]{Aubin:Frankowska:90} by the continuity
assumption of (A2). Further,
Assumption (A3), the bound \eqref{estuNu}
and statement \eqref{subsequence:converge} imply
\begin{equation} \label{fNdistest}
  \begin{aligned}
&\int_0^T\dist(f_N(t),F(t,u(t)))_Hdt\\
&\le 2c_F(\sqrt{T}+K)\|u_N-u\|_{L^2(0,T;H)}\to 0\quad\text{as}\quad\N^{(4)}\ni N\to\infty.
  \end{aligned}
  \end{equation}
Hence for every $\epsilon>0$
there exists $N_0(\epsilon)$ such that 
$f_N\in\F_\epsilon$ for all $N\in N^{(4)}$ with $N\ge N_0(\epsilon)$.
By Mazur's theorem, the set $\F_\epsilon$ is weakly closed in $L^1(0,T;H)$, 
so that \eqref{fweakconv} implies $f\in\F_\epsilon$. 
Since this holds for every $\epsilon>0$, we have proved $f\in\F$
and hence inclusion \eqref{introPDI2}.

Since the Galerkin spaces $V_N$ are nested, we obtain from \eqref{introGDI:all}
and \eqref{weakderiv2} for all $v_M \in V_M$, all $\varphi\in C_c^{\infty}(0,T)$
and all $N \ge M$ that
\begin{equation*} 
\int_0^T(u_N(t),v_M)\varphi'(t)dt=\int_0^T\Big(a(u_N(t),v_M)\varphi(t)
-(f_N(t),v_M)\varphi(t)\Big)dt.
\end{equation*}
Let $N \rightarrow \infty$ in this equality, and 
use the convergences \eqref{subsequence:converge}, \eqref{uNweakstar}
and \eqref{fweakconv} for $v_M \varphi'\in L^2(0,T;H)$,
$v_M \varphi \in L^2(0,T;V)$, $v_M \varphi \in L^{\infty}(0,T;H)$,
to obtain
\begin{equation*} 
\int_0^T(u(t),v_M)\varphi'(t)dt=\int_0^T\Big(a(u(t),v_M)\varphi(t)
-(f(t),v_M)\varphi(t)\Big)dt.
\end{equation*}
Since $\bigcup_{M \in \N} V_M$ is dense in $V$, we have 
\begin{equation*} 
\int_0^T (u(t),v)\varphi'(t) dt = \int_0^T \Big(a(u(t),v)
-(f(t),v)\Big)\varphi(t)dt
\end{equation*}
for all $v \in V$ and $\varphi \in C_c^{\infty}(0,T)$.
Hence $u$ has a weak derivative in $W_+=L^2(0,T,V^*)+ L^1(0,T;H)$,
given by the two terms on the right-hand side. Thus we have shown that $u\in W$ satisfies
statements \eqref{introPDI1} and \eqref{introPDI2}.

\medskip

\noindent\emph{d) The solution assumes the initial data}\\
It remains to verify that the initial value condition 
\eqref{introPDI:IV} is satisfied.
Consider again $N \ge M$ and an arbitrary element $v_M \in V_M$.
Using the weak differentiability of $u$ and the absolute
continuity of $u_N$, we obtain
\begin{align*}
-&(u_{N,0}-u(0),v_M) =  \int_0^T \frac{d}{dt}\Big[ (u_N(t)-u(t),\tfrac{T-t}{T}v_M)
\Big] dt \\
= & \int_0^{T} (u_N'(t),\tfrac{T-t}{T}v_M) - \langle u'(t),\tfrac{T-t}{T}v_M \rangle 
- \frac{1}{T} (u_N(t)-u(t),v_M)dt\\
= & \int_0^T\Big(-a(u_N(t)-u(t),\tfrac{T-t}{T}v_M)+ 
(f_N(t)-f(t),\tfrac{T-t}{T}v_M)\Big) dt \\
- &  \int_0^T \frac{1}{T} (u_N(t)-u(t),v_M)dt. 
\end{align*}
In the last step we used the differential inclusions
\eqref{introPDI1} and \eqref{introGDI1}.
Equation \eqref{subsequence:converge} shows that the last term  
converges to $0$ as $\N''\ni N\to\infty$.
The first two integral terms converge to zero due to the weak$*$-convergence 
\eqref{uNweakstar}, the weak convergence \eqref{fweakconv}
and the fact that the test function $t \mapsto \tfrac{T-t}{T}v_M$
is in $ L^2(0,T;V)\cap L^{\infty}(0,T;H)$.
From assumption (A5), we conclude
\begin{equation*}
(u_0-u(0),v_M)=0\quad\forall\,v_M \in V_M.
\end{equation*}
Now the density of $\bigcup_{M} V_M$ in $H$ shows $u(0)=u_0$.

\medskip

\noindent\emph{e) Upper Kuratowski convergence of $\mc{S}_N$ to $\mc{S}$}\\
If $\Limsup_{N\rightarrow\infty}\mc{S}_N \subseteq \mc{S}$ is false in $L^2(0,T;H)$, 
there exist a subsequence $\N'\subseteq\N$, elements $u_N\in\mc{S}_N$ for all $N\in\N'$, 
and some $u\in L^2(0,T;H)$ with $u\notin\mc{S}$ and $\|u_N-u\|_{L^2(0,T;H)}\rightarrow 0$.
But by the above, there exist a subsequence $\N''\subseteq\N'$ and some $\tilde u\in\mc{S}$
such that $\|u_N-\tilde u\|_{L^2(0,T;H)}\rightarrow 0$ as $N\rightarrow\infty$ in $\N''$,
which is a contradiction.
\end{proof}

\subsection{Lower limit of Galerkin solution sets}

In this section we show that solutions in $\mc{S}$ are limits of
solutions in $\mc{S}_N$ with respect to the norm of 
$L^2(0,T;V)\cap L^{\infty}(0,T;H)$.
In principle, we project every $u\in\mc{S}$ to $W_N$ and
invoke Proposition \ref{Filippov} to generate a Galerkin solution
close to $u$.

\begin{lemma} \label{Zeidler}
For any $u\in W$ and $N\in\N$, there exists a unique solution $v_N\in W_N$ 
of the linear problem
\begin{align*}
&(v_N'(t),v)+a(v_N(t),v)=\langle u'(t),v\rangle+a(u(t),v) 
\quad\dot\forall t\in(0,T),\,\forall v \in V_N,\\
&v_N(0) = P_Nu(0),
\end{align*}
and the sequence $\{v_N\}_{N\in\N}$ of these solutions satisfies
\begin{align} 
\|v_N-u \|_{W_+^{\star}}\to 0\quad\text{as}\quad N\to\infty. \label{QN:pointwise}
\end{align}
\end{lemma}

\begin{proof}
This is the usual Galerkin approximation of $w'+Aw=f$ in $V_N$ 
with $f\in L^2(0,T;V^*)+L^1(0,T;H)$ given by 
$f(t):=u'(t)+Au(t)$ for almost every $t\in(0,T)$. 
Existence and uniqueness of a solution as well as
convergence \eqref{QN:pointwise} are shown in \cite[Theorem 23.A]{Zeidler:2A}.
\end{proof}

Now we prove lower Kuratowski convergence of $\mc{S}_N$ to $\mc{S}$.

\begin{proof}[Proof of Theorem \ref{lower:Kuratowski}]
Let $u\in\mc{S}$ be a solution and let $f\in L^1(0,T;H)$ be as in
\eqref{introPDI1} and \eqref{introPDI2}.
For all $N\in\N$ and $v_N$ as in Lemma \ref{Zeidler}, we have
\begin{align*}
&(v'_N(t),v)+a(v_N(t),v)
=\langle u'(t),v\rangle+a(u(t),v)\\
&=(f(t),v)
=(P_Nf(t),v)\quad \dot\forall\,t\in(0,T),\,\forall\,v\in V_N.
\end{align*}
Therefore, the distance
\[\delta_N(t):=\dist(v_N'(t)+A_Nv_N(t),F(t,v_N(t)))_H\]
satisfies
\begin{align*}
  \delta_N(t) \le \|P_N f(t) - f(t) \|_H + \dist(f(t),F(t,v_N(t)))_H.
\end{align*}
The first term converges to $0$ pointwise by assumption (S2) and has
the integrable bound $\|f(t)\|_H, t \in (0,T)$, hence
$\| P_Nf - f \|_{L^1(0,T;H)} \to 0$ as $N \to \infty$ by Lebesgue's theorem.
The second term is estimated by invoking assumption (A3)
\begin{equation} \label{Fdistest}
\begin{aligned}
& \int_0^T\dist(f(t),F(t,v_N(t)))_H dt \le
\int_0^T\dist(F(t,u(t)),F(t,v_N(t)))_Hdt\\
&\le c_F(T+\|u\|_{L^2(0,T;V)}+\|v_N\|_{L^2(0,T;V)})\|u-v_N\|_{L^2(0,T;H)},
\end{aligned}
\end{equation}
which converges to $0$ as $N \to \infty$ due to \eqref{QN:pointwise}.
Hence, by Proposition~\ref{Filippov}, there exist solutions
$u_N\in W_N (N\in\N)$ of inclusion \eqref{introGDI:all}, 
with $u_{N,0}= P_N u(0)$ which satisfy
\begin{align*}
\|u_N-v_N\|_{W_+^*} \le C_{\ell} \|\delta_N\|_{L^1(0,T)}
\to 0\quad\text{as}\quad N\to\infty.
\end{align*}
Finally, statement \eqref{QN:pointwise} implies
\begin{equation*}
\|u-u_N\|_{W_+^{\star}} \le \|u-v_N\|_{W_+^{\star}} + \|v_N-u_N\|_{W_+^{\star}}  
\to 0\quad\text{as}\quad N\to\infty.
\end{equation*}
\end{proof}

\begin{remark*}
Note that  convergence in Theorem \ref{lower:Kuratowski}
holds in the strong norm of $L^2(0,T;V)\cap L^{\infty}(0,T;H)$,
 while Theorem \ref{upper:Kuratowski} assures convergence only in the
weaker norm of $L^2(0,T;H)$ (or in $L^r(0,T;H)$ with a fixed $r \in [1,\infty)$). 
\end{remark*}

\subsection{Hausdorff convergence of solution sets}

We finally combine upper and lower Kuratowski convergence to prove convergence 
with respect to the Hausdorff metric in $L^2(0,T;H)$.

\begin{proof}[Proof of Theorem  \ref{main:theorem}]
Theorems \ref{upper:Kuratowski} and \ref{lower:Kuratowski} imply Kuratowski convergence
$\Lim_{N\rightarrow\infty}\mc{S}_N=\mc{S}$ in $L^2(0,T;H)$.
According to Proposition \ref{aprioribound}, 
there exists a ball $B\subset W$ with $\mc{S}\subset B$ and $\mc{S}_N\subset B$ 
for all $N\in\N$.
Therefore, invoking Lemma \ref{Kuratowski:lemma} with $Y=L^2(0,T;H)$ yields 
the desired convergence statement, because $B$ is relatively compact in $L^2(0,T;H)$ 
by the compact embedding \eqref{compactembedW}.
\end{proof}

\subsection{An extension}
\label{sec:extension}
Condition (A3) seriously limits the polynomial growth of nonlinearities
to which Theorems \ref{upper:Kuratowski} and \ref{lower:Kuratowski}
apply, see Section \ref{sec:example}. Therefore, we present in this section a weaker assumption
under which our conclusions still hold. The type of condition
is similar to \cite[Assumption B]{emmrich2009} where it is used
to derive a priori estimates for variable time-step discretizations of
evolution equations.

Instead of (A3) we require \\
(A3') There exist functions $\alpha\in L^1(0,T)$, $b \in C(\R)$
and constants $\beta \in [0,2)$, $\gamma \in (0,1]$ 
such that for almost every $t\in(0,T)$ and all $u,v \in V$, the following
estimates hold
\begin{align*} 
&\|F(t,0)\|_H\le\alpha(t),\\
 & \dist_{\h}(F(t,u),F(t,v))_{H}\le b(\|u\|_H+\|v\|_H)
  (1+\|u\|_V^{\beta}+\|v\|_V^{\beta})\|u-v\|_H^{\gamma} .
\end{align*}
For $\beta=\gamma=1$ and a constant function $b$ this coincides with condition (A3).
Without loss of generality we assume the function $b$ to be
nonnegative and monotone increasing.
In the following we indicate the changes in the proofs of our results without
providing all details.

The growth estimate \eqref{L1bound} now reads
\begin{equation} \label{newL1bound}
  \| F(t,v) \|_H \le \alpha(t) + b(\|v\|_H)(1+\|v\|_V^{\beta}) \|v\|_H^{\gamma}
  \quad \dot\forall t \in (0,T), \; \forall v \in  V.
\end{equation}
The modified $L^1$-estimate \eqref{fNest} is obtained from \eqref{newL1bound},
\eqref{Gronwall1} and \eqref{L2vestimate}  via  H\"older's inequality as follows
\begin{equation} \label{newfNest}
  \|f_N \|_{L^1(0,T;H)} \le \|\alpha\|_{L^1(0,T)} + b(K_1)K_1^{\gamma}(T +
  K_0^{\beta}T^{1 - \frac{\beta}{2}}).
  \end{equation}
In a similar way, the $L^1$-bound and uniform integrability of the
sequence $f_N \in L^1(0,T;H)$ from part b) of the proof of  Theorem
\ref{upper:Kuratowski} follows from \eqref{newL1bound}
\begin{equation*} 
  \sup_{N \in \N'''} \| \int_E f_N(t) dt \|_H \le
  \int_E \alpha(t) dt + b(K) K^{\gamma}(\mu(E) + K^{\beta}\mu(E)^{1 - \frac{\beta}{2}}).
\end{equation*}
Moreover, since the compact embedding \eqref{compactembedW} holds for
all $1 \le r < \infty$ we select the subsequence $\N''$ such that
 one has, instead of \eqref{subsequence:converge}
\begin{equation} \label{newsubsequence:converge}
  \|u_N-u\|_{L^r(0,T;H)}\to 0\quad\text{as}\quad\N''\ni N\to\infty, \quad
  r:= \frac{2 \gamma}{2 - \beta}.
\end{equation}
Using H\"older's inequality 
the estimate \eqref{fNdistest} is replaced by
\begin{equation*} 
  \begin{aligned}
& \int_0^T\dist(f_N(t),F(t,u(t)))_Hdt  \\  
    \le &\, b(2 K) \left\{ \int_0^T (1 + \|u_N(t)\|_V^{\beta}+\|u(t)\|_V^{\beta})^{\frac{2}{\beta}} dt \right\}^{\frac{\beta}{2}}\|u_N - u\|_{L^r(0,T;H)}^{\gamma}  \\
    \le  & \, C b(2K) (1 + \|u_N \|_{L^2(0,T;V)} + \|u(t)\|_{L^2(0,T;V)})^{\beta}
    \|u_N - u\|_{L^r(0,T;H)}^{\gamma}
\end{aligned}
\end{equation*}
with a suitable constant $C$. Finally, with $r$ from
\eqref{newsubsequence:converge}, the estimate \eqref{Fdistest} is modified
in a similar way:
\begin{equation*}
  \begin{aligned}
    & \int_0^T \dist(F(t,u(t)),F(t,v_N(t))_H dt \\
    \le & C b(2 \|u\|_{W_+^*}) (1+ \|u\|_{L^2(0,T;V)} +
    \|v_N\|_{L^2(0,T;V)})^{\beta} \|u - v_N\|_{L^r(0,T;H)}^{\gamma}.
  \end{aligned}
  \end{equation*}
    
\section{Example}
\label{sec:example}
Let $\Omega\subset\R^1$ be a bounded open 
interval,
and consider the Gelfand triple $V\subseteq H\subseteq V^*$ 
with spaces $V=H^1_0(\Omega)$, $H=L^2(\Omega)$ and $V^*=H^{-1}(\Omega)$,
which satisfy assumption (A1).
Let $(V_N)_{N\in\N}\subset V$ be spaces of piecewise linear functions
subject to successively refined equidistant grids.
This standard construction is known to satisfy (S2) and under some additional conditions also satisfies (S3), see \cite{carstensen02}.

Consider the partial differential inclusion
\[u'(t)-\Delta u(t)\in F(u(t)).\]
It is well-known that $-\Delta:V\to V^*$ induces a bilinear form which 
satisfies assumption (A1).
By the Sobolev embedding theorem, we have $V\subset L^\infty(\Omega)$, 
and there exists $C_\infty>0$ with
\[\|v\|_{L^\infty(\Omega)}\le C_\infty\|v\|_V\quad\forall\,v\in V.\]
The function
\[g:\R\to\R,\quad g(\eta)=\eta(1-|\eta|) \quad\forall\,\eta\in\R,\] 
clearly does not have linear growth.
It satisfies
\[(\int_\Omega g(v(x))^2dx)^\frac12
\le \|v\|_{L^4(\Omega)}^2 + \|v\|_{L^2(\Omega)},\]
and hence induces a Nemytskii operator 
\[N_g:L^4(\Omega)\to L^2(\Omega),\quad N_g(v)(x):=g(v(x)).\]
In view of \cite[Theorem 3.4.4]{Gasinski:Papageorgiou:06}, the operator
$N_g$ is continuous.
Let $h\in L^2(\Omega)$ be nonnegative.
Then the multivalued nonlinearity
\begin{align*}
  F &:[0,T]\times V\to\CB(H), \\
  F(t,v)& =\{f\in \M(\Omega;\R): |N_g(v)(x)- f(x)| \le h(x)
 \; \dot\forall \; x\in \Omega \}
\end{align*}
 does not depend on $t$ and clearly satisfies the
measurability condition of (A2). Continuity will follow from
(A3) which we verify next.
Since
\[\|F(t,0)\|_{L^2(\Omega)}=  \|h\|_{L^2(\Omega)},\]
the first inequality holds with $\alpha(t)\equiv\ \|h\|_{L^2(\Omega)}$.
For $\xi,\eta\in\R$ we have
\[|g(\xi)-g(\eta)|= |\xi-\eta+|\xi|(\eta-\xi)+\eta(|\eta|-|\xi|)| \le (1+|\xi|+|\eta|)|\xi-\eta|,
\]
hence for $u,v \in L^2(\Omega)$
\begin{align*}
&\dist(F(t,u),F(t,v))_H^2
\le\int_{\Omega}|g(u(x))-g(v(x))|^2dx\\
&\le\int_{\Omega}(1+|u(x)|+|v(x)|)^2(u(x)-v(x))^2dx\\
&\le(1+\|u\|_{L^\infty(\Omega)}+\|v\|_{L^\infty(\Omega)})^2\|u-v\|_{L^2(\Omega)}^2\\
& \le \max(1,C_\infty)^2(1+\|u\|_V+\|v\|_V)^2\|u-v\|_H^2.
\end{align*}
Therefore, the second inequality of assumption (A3) holds with \newline $c_F=\max(1,C_\infty)$.
One easily verifies that
\[(g(\xi)-g(\eta))(\xi-\eta)\le (\xi-\eta)^2\quad\forall\,\xi,\eta \in\R,\]
and it follows that
\begin{align*} (N_g(u)-N_g(v),u-v)&=\int_\Omega (g(u(x))-g(v(x)))(u(x)-v(x))dx\\
& \le \int_{\Omega}(u(x)-v(x))^2 dx.
\end{align*}
This implies assumption (A4) with $\ell(t)\equiv 1$.
All in all, Theorem \ref{main:theorem} applies and yields
\[\dist_{\h}(\mc{S},\mc{S}_N)_{L^2(0,T;H)}\rightarrow 0\quad\text{as}\quad N\rightarrow\infty.\]
Let us finally note that the weakened assumption (A3') from
section~\ref{sec:extension} allows to treat more general nonlinearities,
for example
\begin{equation*}
  g(\eta) = \eta(1 - |\eta|^{2- \varepsilon}) \quad \text{for some} \;
  0< \varepsilon \le 2.
\end{equation*}
However, it remains as an open problem whether our results extend to
the standard cubic nonlinearity with $\varepsilon =0$.

\end{document}